\documentclass[a4paper,12pt]{amsart}
\usepackage[pdftex]{color,graphicx}
\usepackage{amssymb}
\usepackage{amsfonts}
\usepackage{esint}
\usepackage{relsize}
\usepackage{amsmath}
\usepackage{amsrefs}
\usepackage{epsfig}
\usepackage{amsthm}
\usepackage{color}
\usepackage[]{epsfig}
\usepackage[]{pstricks}
\usepackage[utf8]{inputenc}
\usepackage{multicol}
\usepackage{accents}
\newtheorem{theorem}{Theorem}[section]
\newtheorem{proposition}{Proposition}[section]
\newtheorem{lemma}{Lemma}[section]
\newtheorem{corollary}{Corollary}[section]

\theoremstyle{remark}
\newtheorem{remark}{Remark}[section]
\theoremstyle{definition}

\newcommand{\real}{\mathbb{R}}
\newcommand{\s}{\mathbb{S}}

\newcommand{\LL}{\mathcal{L}}
\newcommand{\n}{\nabla}
\newcommand{\ran}{\rangle}
\newcommand{\lan}{\langle}
\newcommand{\ve}{\varepsilon}
\newcommand{\vp}{\varphi}
\DeclareMathOperator{\hess}{Hess}
\DeclareMathOperator{\ric}{Ric}
\DeclareMathOperator{\di}{div}
\DeclareMathOperator{\tr}{trace}
\DeclareMathOperator{\grad}{grad}

\DeclareMathOperator{\dist}{dist}
\DeclareMathOperator{\proj}{proj}

\numberwithin{equation}{section}
\title[Self-shrinkers tangent spaces omit a set]{Rigidity of complete self-shrinkers whose tangent planes omit a nonempty set}

\author{Hil\'ario Alencar, Manuel Cruz \& Greg\'orio Silva Neto}

\date{April 15, 2023}

\address{Instituto de Matemática, Universidade Federal de Alagoas, Macei\'o, AL, 57072-900, Brasil}

\email{hilario@mat.ufal.br}

\address{Instituto de Matemática, Universidade Federal de Alagoas, Macei\'o, AL, 57072-900, Brasil}

\email{mbrutusc@gmail.com}

\address{Instituto de Matemática, Universidade Federal de Alagoas, Macei\'o, AL, 57072-900, Brasil}
\email{gregorio@im.ufal.br}

\begin{document}

 
\footnotetext{Hil\'ario Alencar is the corresponding author. E-mail: hilario@mat.ufal.br.\\ 
Hil\'ario Alencar and Greg\'orio Silva Neto were partially supported by the National Council for Scientific and Technological Development - CNPq of Brazil. Manuel Cruz was partially supported by Coordena\c c\~ao de Aperfei\c coamento de Pessoal de N\'ivel Superior - Brasil (CAPES) - Finance Code 001.}

\begin{abstract}
In this paper we prove rigidity results for the sphere, the plane and the right circular cylinder as the only self-shrinkers satisfying a classic geometric assumption, namely the union of all tangent affine submanifolds of a complete self-shrinker omits a non-empty set of the Euclidean space. This assumption lead us to a new class of submanifolds, different from those with polynomial volume growth or the proper ones. We also prove an analogous result for self-expanders.
\end{abstract}
\maketitle

\section{Introduction}

A $n$-dimensional submanifold $X:\Sigma^n\to\real^{n+k}, n\geq2,\ k\geq1,$ is called a self-shrinker if it satisfies
\[
{\bf H}=-\frac{1}{2} X^\perp,
\]
where ${\bf H}=\sum_{i=1}^n \alpha(e_i,e_i)$ is the mean curvature vector field of $\Sigma^n$ and $X^\perp$ is the part of $X$ normal to $\Sigma^n.$ 

Self-shrinkers are self-similar solutions of the mean curvature flow and plays an important role in the study of this flow since they are type I singularities of the flow, see \cite{Colding-Minicozzi}. The simplest examples of self-shrinkers are the round spheres, planes and cylinders. Moreover, there are many results which present these examples as the only self-shrinkers satisfying some geometric restrictions, see \cite{Colding-Minicozzi}, \cite{Cao-Li}, \cite{AS}, \cite{DX}, and \cite{DXY}. In commom, all these results have the assumption that, when the self-shrinker is not compact, it must have polynomial volume growth or it must be proper. In \cite{CZ}, Cheng and Zhou proved that a self-shrinker has polynomial volume growth if and only if it is proper. Recently, joinly with Vieira, see \cite{CVZ}, they generalized this result for submanifolds with bounded weighted mean curvature in a wide class of shrinking gradient Ricci solitons, which includes the Gaussian soliton. In particular, Cheng-Vieira-Zhou result gives that, for a surface with bounded ${\bf H}+\frac{1}{2} X^\perp,$ polynomial volume growth is equivalent to the properness of the submanifold (see Theorems 1.3 and 1.4 of \cite{CVZ}).

In this paper, we prove the rigidity of the sphere, the cylinders and the affine subspaces passing through the origin as the only self-shrinkers under another classic geometric assumption we describe below. Here and elsewhere, we identify the tangent spaces $T_p\Sigma^n$ with the affine subspace $X(p)+dX_p(T_p\Sigma^n),$ tangent to $X(\Sigma)$ at $X(p)$.

Let us denote by 
\[
W=\real^{n+k}\backslash\bigcup_{p\in\Sigma^n} T_p\Sigma^n
\]
the set omitted by the union of the affine subspaces tangent to $X(\Sigma^n)\subset\real^{n+k}$. Here, we purpose to classify the self-shrinkers with nonempty $W.$ The study of submanifolds of the Euclidean space with non-empty $W$ started with Halpern, see \cite{halpern}, who proved that compact and oriented hypersurfaces of the Euclidean space have nonempty $W$ if and only if it is embedded, diffeomorphic to the sphere and it is the boundary of a star-shaped domain of $\real^{n+1}.$ Therefore, since the only self-shrinker with these characteristics are the round spheres of radius $\sqrt{2n}$ (see \cite{Huisken}), the case of compact self-shrinkers of codimension one with nonempty $W$ is completely solved. 

In the non-compact case there are many examples of hypersufaces with nonempty $W.$ In fact, cylinders and paraboloids have open and nonempty $W$ and the one sheet hyperboloid has $W=\{0\}$. Surprisingly, in dimension two, if $\Sigma^2$ is a minimal surface of $\real^3,$ then Hasanis and Koutrofiotis, see \cite{H-K}, proved that $W\neq\emptyset$ if and only if $\Sigma^n$ is a plane (in fact the result holds for arbitrary codimension, provided $X^\perp/\|X^\perp\|$ is paralell at the normal bundle). For higher dimensions, Alencar and Frensel, see \cite{A-F}, proved that the same result holds in higher dimension hypersurfaces of $\real^{n+1}$ (i.e., $\Sigma^n$ is a hyperplane), assuming in addition that $W$ is open. Other rigidity results involving $W$ and geometric assumptions can be found in \cite{AB}, \cite{BMR}, \cite{MM} and \cite{PRS}.  

First, we present two rigidity results for complete $n$-dimensional self-shrinkers in $\real^{n+1}$ with open and nonempty $W$.

\begin{theorem}\label{theo-1}
Let $\Sigma^n$ be a complete, $n$-dimensional, self-shrinker of $\real^{n+1}.$ If the set $W$ is open and nonempty, and the squared matrix norm $\|A\|^2$ of the second fundamental form $A$ of $\Sigma^n$ satisfies
\[
\|A\|^2\leq \frac{1}{2},
\] 
then $\Sigma^n=\s^{p}(\sqrt{2p})\times\real^{n-p},$ $0\leq p \leq n.$ 
\end{theorem}

If $\|A\|^2\geq 1/2$ and, additionally, we assume the mean curvature $H\geq 0,$ we have

\begin{theorem}\label{theo-2-1}
Let $\Sigma^n$ be a complete, $n$-dimensional, self-shrinker of $\real^{n+1}.$ If the set $W$ is open, $0\in W$, and the squared matrix norm $\|A\|^2$ of the second fundamental form $A$ of $\Sigma^n$ satisfies
\[
\|A\|^2\geq \frac{1}{2},
\]
then $\Sigma^n=\s^{p}(\sqrt{2p})\times\real^{n-p},$ $1\leq p \leq n.$
\end{theorem}

\begin{remark}
Analyzing geometrically, one can see that self-shrinkers of the form $\Sigma^n=\Gamma\times\real^{n-1},$ where $\Gamma$ is a closed Abresch-Langer curve (see \cite{A-L} and \cite{Hall}) satisfies $W$ open and $0\in W$. But all these examples satisfies
\[
\min \|A\|^2<\frac{1}{2}<\max\|A\|^2
\]
after normalization.
\end{remark}

\begin{remark} 
The bound $1/2$ seens natural for self-shrinkers. In \cite{Cao-Li}, Cao and Li proved that the only complete $n$-dimensional self-shrinkers of $\real^{n+k}$ with polynomial volume growth and such that $\|A\|^2\leq1/2$ are $\s^{p}(\sqrt{2p})\times\real^{n-p},$ $0\leq p \leq n.$ Cheng and Peng, see \cite{C-P}, and Rimoldi, see \cite{R}, with the aim to remove the hypothesis of polyminal volume growth, proved that the only self-shrinker of $\real^{n+1}$ with $\sup_{\Sigma^n} \|A\|^2=1/2$ (but with $\|A\|^2<1/2)$ is a hyperplane. On the other hand, there are other rigidity results where the bound $\|A\|^2\geq 1/2$ appears see, for example, \cite{Li-Wei}, \cite{DX}, \cite{Cheng-Wei}, and \cite{X-X}. Again, in common, all the last four references assumes polynomial volume growth or the immersion is proper. 
\end{remark}

\begin{remark}
Submanifolds with $W\neq\emptyset$ is a class of submanifolds distinct from those with polynomial of volume growth or those which are proper. In fact, cylinders over curves $\Gamma$ in $\real^2$ parametrized by $\Gamma(t)=b(t)(\cos t,\sin t),$ where $b(t)=1+e^{-t},$ or the family
\begin{equation}\label{arctan}
b(t)=d+\frac{m}{\pi}\left(\frac{\pi}{2}-\arctan(at)\right),\ d>0,\ m>0, \ 0<a\leq 1.
\end{equation}
Some straightforward calculation can prove that $\Sigma^n=\Gamma\times\real^{n-1}$ satisfies $W=\mathbb{D}^2(d)\times\real,$ and thus, nonempty. Here $\mathbb{D}^2=\{(x,y)\in\real^2;x^2+y^2\leq d^2\}$ is the closed disk of radius $d$. Moreover, this hypersurface is non proper, since it is asymptotic to the cylinder of radius $d$ (and, thus, with volume growth bigger than polynomial, by the results of Cheng-Vieira-Zhou, see \cite{CVZ}). Since the curvature of $\Gamma$ lies in the interval $((d+m)^{-1},d^{-1}),$ we can chose these hypersurfaces satisfying the assumptions on $\|A\|^2$ of Theorems \ref{theo-1} and \ref{theo-2-1}. 
\end{remark}

If $\Sigma^2$ has dimension two, then we can consider arbitrary codimension. We will assume further that ${\bf H}\neq 0$ and that ${\bf H}/\|{\bf H}\|$ is parallel at the normal bundle.

\begin{theorem}\label{theo-2}
Let $\Sigma^2$ be a complete, two-dimensional, self-shrinker of $\real^{2+k},$ $k\geq2,$ with mean curvature vector ${\bf H}\neq 0$ and such that ${\bf H}/\|{\bf H}\|$ is parallel at the normal bundle. If the set $W$ is open and nonempty and the squared matrix norm $\|A\|^2$ of the second fundamental form $A$ of $\Sigma^n,$ relative to ${\bf H}/\|{\bf H}\|,$ satisfies one of the following conditions:
\begin{itemize}
\item[i)] $\|A\|^2\leq 1/2;$
\item[ii)] $\|A\|^2\geq 1/2$ and $0\in W;$
\end{itemize} 
then $\Sigma^2=\s^{p}(\sqrt{2p})\times\real^{2-p},$ $1\leq p \leq 2.$ 
\end{theorem}

\begin{remark}
If we assume that $\Sigma^2$ is compact without boundary in Theorem \ref{theo-2}, then the hypothesis that $W$ is open can be removed. We point out that Smoczyk, see \cite{Smoczyk}, proved that the only compact self-shrinkers, without boundary, of $\real^{n+p}$ with ${\bf H}\neq 0$ and ${\bf H}/\|{\bf H}\|$ parallel in the normal bundle are minimal surfaces of the sphere $\s^{n+p-1}(\sqrt{2n}).$
\end{remark}

\begin{remark}
Drugan and Kleene in \cite{D-K} proved the existence of infinitely many rotational self-shrinkers of each topological type of $\s^n,$ $\s^{n-1}\times\s^1,$ $\real^n$ and $\s^{n-1}\times\real.$ Analyzing geometrically the picture of the profile curves presented there, we can see that the rotational self-shrinkers obtained by the rotation of those profiles curves have empty $W.$ We also remark that all these examples are not embedded since Kleene and M\o ller, see \cite{K-M}, proved that the sphere of radius $\sqrt{2n},$ the plane, and the right cylinder of radius $\sqrt{2(n-1)}$ are the only embedded rotational self-shrinkers of their respective topological type. 
\end{remark}

We conclude this paper with a non existence result for self-expanders with $W$ open and nonempty. Recall that a $n$-dimensional submanifold $X:\Sigma^n\to\real^{n+k}$ is called a self-expander if it satisfies
\[
{\bf H} = \frac{1}{2}X^\perp.
\]

\begin{theorem}\label{theo-4}
There is no complete, non compact, $n$-dimensional self-expanders of $\real^{n+k},$ $k\geq1,$ with principal normal vector field ${\bf H}/\|{\bf H}\|$ parallel in the normal bundle, and such that the set $W$ is open and $0\in W$. 
\end{theorem}

\begin{remark}
Clearly, if the codimension $k=1,$ the hypothesis that ${\bf H}/\|{\bf H}\|$ is parallel in the normal bundle is automatically satisfied and can be omitted in the statement of Theorem \ref{theo-4}.
\end{remark}

\begin{remark}
Analyzing geometrically, one can see that self-expanders of the form $\Sigma^n=\Gamma\times\real^{n-1},$ where $\Gamma$ is a self-expanding curve classified by Halldorsson, see \cite{Hall}, satisfies $H>0$ and $W$ is open and nonempty, but $0\not\in W$, which implies that the hypothesis of $0\in W$ is crucial for the validity of Theorem \ref{theo-4}.
\end{remark}

\begin{remark}
It is well known, see \cite{Cao-Li}, that there is no compact self-expanders in $\real^{n+k}.$ Thus, Theorem \ref{theo-4} does not make sense for $\Sigma^n$ compact.
\end{remark}

\section{Preliminaries}

Let $i:\Sigma^n\to M^{n+k},\ n\geq 2,\ k\geq1,$ be an isometric immersion, where $\Sigma^n$ and $M^{n+k}$ are Riemannian manifolds and the superscripts denote the dimension. Denote by $\n$ and $\overline{\n}$ be the connections of $\Sigma^n$ and $M^{n+k},$ respectively. We assume here that the immersion admits a conformal vector field, i.e., a vector field $X\in TM$ such that
\begin{equation}\label{conf}
\overline{\n}_YX = \vp Y,
\end{equation}
for some smooth function $\vp:M^{n+k}\to\real,$ called conformal factor of $X,$ and for every $Y\in T\Sigma^n.$ Decompose $X$ as
\[
X=X^\top+X^\perp,
\]
where $X^\top\in T\Sigma^n$ and $X^\perp\in (T\Sigma^n)^\perp.$ Here $(T\Sigma^n)^\perp$ is the normal bundle of the immersion such that $T\Sigma^n\oplus (T\Sigma^n)^\perp=TM^{n+k}.$ 

If the codimension is one, then we have $X^\perp = \lan X,N\ran N,$ where $N$ is the globally defined unitary normal vector field. If the codimension is at least two, suppose further that $X^\perp\neq 0.$ In both cases we can write


\begin{equation}\label{x-decomp}
X=X^\top + f\eta,
\end{equation}
for $f=\lan X,\eta\ran,$ where $\eta=N$ if the codimension is one and $\eta=X^\perp/\|X^\perp\|$ if the codimension is at least two. 


The immersion satisfies
\begin{equation}\label{G-W}
\overline{\n}_UV = \n_UV + \alpha(U,V) \ \mbox{and} \ \overline{\n}_U\eta=-AU + \n^\perp_U\eta,
\end{equation}
where $\lan AU,V\ran=\lan\alpha(U,V),\eta\ran,$ $\alpha$ is the second fundamental form of the immersion, and $\n^\perp$ denotes the normal connection at the normal bundle $(T\Sigma^n)^\perp$. 

The next proposition contains the basic calculations needed to prove the main theorems of this paper.

\begin{proposition}\label{L-prop-0}
Let $M^{n+k}$ be a $(n+k)$-dimensional Riemannian manifold which admits a conformal vector field $X$ with conformal factor $\vp.$ Let $\Sigma^n$ be a submanifold of $M^{n+k}$ and $\{\eta,\eta_2,\ldots,\eta_k\}$ be an orthonormal frame of the normal bundle $(T\Sigma^n)^\perp\subset TM^{n+k},$ where, for $k=1,$ $\eta=N,$ the globally defined unitary normal vector field, and for $k\geq2,$ we assume that $X^\perp\neq 0$ and take $\eta=X^\perp/\|X^\perp\|.$
If $f=\lan X,\eta\ran,$ then
\begin{equation}\label{lap-f0}
\begin{split}
\Delta f&  +\vp(\tr A) + f\|A\|^2 + \lan X^\top,\grad(\tr A)\ran\\
&=- \sum_{\beta=2}^k s_{1\beta}(A_\beta X^\top) + \sum_{\beta=2}^k s_{1\beta}(X^\top)(\tr A_\beta)\\
&\qquad + \sum_{i=1}^n\lan\overline{R}(e_i,X^\top)e_i,\eta\ran.
\end{split}
\end{equation}
Here, $A$ and $A_\beta$ are the shape operators relative to the normals $\eta$ and $\eta_\beta,$ $\beta\in\{2,\ldots,k\},$ respectively, $\|A\|^2=\tr (A^2)$ is the matrix norm of $A,$ $s_{1\beta}(X)=\lan\n^\perp_X\eta,\eta_\beta\ran,$ $\overline{R}$ is the curvature tensor of $M^{n+k},$ and $\{e_1,e_2,\ldots,e_n\}$ is an orthonormal frame of $\Sigma^n.$ 

If the immersion has codimension one (i.e., $k=1$), then
\begin{equation}\label{lap-f}
\begin{split}
\Delta f +\vp H + f\|A\|^2 + \lan X^\top,\grad H\ran = \sum_{i=1}^n\lan\overline{R}(e_i,X^\top)e_i,\eta\ran,
\end{split}
\end{equation}
where $H=\tr A$ is the mean curvature of $\Sigma^n.$ In particular, if the Ricci curvature of $M^{n+1}$ is constant (i.e., $M^{n+1}$ is an Einstein space), then
\begin{equation}\label{lap-f00}
\begin{split}
\Delta f +\vp H + f\|A\|^2 + \lan X^\top,\grad H\ran = 0.
\end{split}
\end{equation}
\end{proposition}

\begin{proof}
Let $U\in T\Sigma^n.$ Since, using (\ref{G-W}),
\[
\begin{split}
\vp U &= \overline{\n}_UX = \overline{\n}_UX^\top + (Uf)\eta + f\overline{\n}_U\eta\\
      &= \n_U X^\top + \alpha(X^\top,U) + (Uf)\eta - fAU + f\n^\perp_U \eta,    
\end{split}
\]
we have, taking the tangent and the normal parts,
\begin{equation}\label{x-tan}
\vp U = \n_UX^\top - fAU
\end{equation}
and
\begin{equation}\label{x-norm}
\alpha(X^\top,U) + (Uf)\eta + f\n^\perp_U\eta=0.
\end{equation}
From (\ref{x-tan}) we have
\begin{equation}\label{div-0}
\n_U X^\top = (\vp I + fA)U,
\end{equation}
which implies
\begin{equation}\label{x-tan-2}
\di X^\top= n\vp + f(\tr A),
\end{equation}
where $\di X^\top$ is the divergence of $X^\top$ in $\Sigma^n.$ From (\ref{x-norm}) we obtain
\begin{equation}\label{Uf}
Uf = -\lan\alpha(X^\top,U),\eta\ran,
\end{equation}
since $\lan \n^\perp_U \eta,\eta\ran=0.$ Therefore,
\begin{equation}
\grad f = - AX^\top.
\end{equation}
Let $\{\eta_1=\eta,\eta_2,\ldots,\eta_k\}$ be an orthonormal frame of $(T\Sigma^n)^\perp$ and write
\[
\n^\perp_U\eta = \sum_{\beta=2}^k s_{1\beta}(X)\eta_\beta,\ \mbox{where}\ s_{1\beta}(X)=\lan\n_X^\perp \eta,\eta_\beta\ran.
\]
Taking the inner product of (\ref{x-norm}) with $\eta_\beta,$ we have
\[
\lan\alpha(X^\top,U),\eta_\beta\ran + fs_{1\beta}(U)=0
\]
i.e.,
\begin{equation}\label{x-norm-2}
fs_{1\beta}(U)=-\lan A_{\beta}X^\top,U\ran,
\end{equation}
where $\lan A_\beta U,V\ran=\lan\alpha(U,V),\eta_{\beta}\ran.$ 

Let us calculate the Laplacian of $f.$ Since, by (\ref{Uf}), $Uf = -\lan AX^\top,U\ran,$ and using (\ref{div-0}), we obtain
\[
\begin{split}
U(Uf)&=-U\lan AX^\top,U\ran = -U\lan X^\top,AU\ran\\
     &=-\lan \n_U X^\top,AU\ran - \lan X^\top,\n_U(AU)\ran\\
     &=-\vp\lan U,AU\ran - f\lan AU,AU\ran - \lan X^\top,\n_U(AU)\ran\\
\end{split}
\]
and $(\n_UU)f = -\lan AX^\top,\n_UU\ran = -\lan X^\top,A(\n_UU)\ran.$ This implies
\[
\hess f(U,U) = -\vp\lan U,AU\ran - f\lan AU,AU\ran - \lan X^\top,(\n_UA)(U)\ran,
\]
where $(\n_U A)(V)=\n_U AV - A(\n_UV).$ Taking the trace, we have
\[
\Delta f=-\vp(\tr A) - f(\tr (A^2)) - \sum_{i=1}^n\lan X^\top(\n_{e_i}A)(e_i)\ran,
\]
where $\{e_1,e_2,\ldots,e_n\}$ is an orthonormal frame of $T\Sigma^n.$ On the other hand, the Codazzi equation
\begin{equation}\label{codazzi}
\begin{aligned}
\lan \overline{R}(U,V)W,\eta\ran &= \lan(\n_VA)(U) - (\n_UA)(V),W\ran\\
&\quad + \lan \alpha(V,W),\n^\perp_U\eta\ran - \lan\alpha(U,W),\n^\perp_V\eta\ran
\end{aligned}
\end{equation}
and
\[
\lan \alpha(V,W),\n_U^\perp \eta\ran = \sum_{\beta=2}^k s_{1\beta}(U)\lan \alpha(V,W),\eta_\beta\ran = \sum_{\beta=2}^k s_{1\beta}(U)\lan A_\beta V,W\ran
\]
give
\begin{equation}\label{codazzi-2}
\begin{aligned}
(\n_UA)(V) &= (\n_VA)(U) + \sum_{\beta=2}^k[s_{1\beta}(U)A_{\beta}V - s_{1\beta}(V)A_{\beta}U]\\
&\qquad + \sum_{i=1}^n\lan\overline{R}(U,V)\eta,e_k\ran e_k.
\end{aligned}
\end{equation}
Since $A$ is symmetric, $\n_UA$ is symmetric also, and moreover 
\[
\tr(\n_UA)=U(\tr A).
\] 
These equations give
\[
\begin{aligned}
\sum_{i=1}^n\lan X^\top,(\n_{e_i}A)(e_i)\ran&=\sum_{i=1}^n\lan(\n_{e_i}A)(X^\top),e_i\ran=\sum_{i=1}^n\lan(\n_{X^\top}A)(e_i),e_i\ran \\
&\qquad + \sum_{\beta=2}^k\sum_{i=1}^n[s_{1\beta}(e_i)\lan A_\beta X^\top,e_i\ran - s_{1\beta}(X^\top)\lan A_\beta e_i,e_i\ran]\\
&\qquad\qquad + \sum_{i=1}^n\lan \overline{R}(e_i,X^\top)\eta,e_i\ran\\
&=\tr(\n_{X^\top}A) + \sum_{\beta=2}^k\sum_{i=1}^n s_{1\beta}(e_i)\lan A_{\beta}X^\top,e_i\ran\\
&\qquad - \sum_{\beta=2}^k s_{1\beta}(X^\top)(\tr A_\beta) + \sum_{i=1}^n\lan\overline{R}(e_i,X^\top)\eta,e_i\ran\\
&=\lan X^\top,\grad(\tr A)\ran + \sum_{\beta=2}^k s_{1\beta}(A_\beta X^\top)\\
&\qquad- \sum_{\beta=2}^k s_{1\beta}(X^\top)(\tr A_\beta)  + \sum_{i=1}^n\lan\overline{R}(e_i,X^\top)\eta,e_i\ran,\\
\end{aligned}
\]
which implies
\begin{equation}
\begin{split}
\Delta f &= -\vp(\tr A) - f\|A\|^2 - \lan X^\top,\grad(\tr A)\ran\\
&\qquad - \sum_{\beta=2}^k s_{1\beta}(A_\beta X^\top) + \sum_{\beta=2}^k s_{1\beta}(X^\top)(\tr A_\beta) + \sum_{i=1}^n\lan\overline{R}(e_i,X^\top)e_i,\eta\ran,
\end{split}
\end{equation}
where $\|A\|^2=\tr (A^2)$ is the matrix norm of $A.$
\end{proof}


In the next consequence of Proposition \ref{L-prop-0}, let us assume that there exists $\ve\in\real$ such that, restricted to $\Sigma^n,$ 
\begin{equation}\label{soliton}
{\bf H}=\ve X^\perp,
\end{equation}
where ${\bf H}=\sum_{i=1}^n \alpha(e_i,e_i)$ is the mean curvature vector field of $\Sigma^n$ in $M^{n+k}.$ If $M^{n+k}=\real^{n+k},$ then $\Sigma^n$ is a mean curvature flow soliton, which is called a self-shrinker, if $\ve<0,$ and a self-expander, if $\ve>0.$ Here, we will adopt one of the canonical normalizations, considering $\ve=-\frac{1}{2}$ for self-shrinkers and $\ve=\frac{1}{2}$ for self-expanders.



If $\Sigma^n$ is submanifold of $M^{n+k}$ satisfying (\ref{soliton}), then
\begin{equation}\label{trace1}
\tr A = \ve f \ \mbox{and}\ \tr A_\beta = 0.
\end{equation}

Let us define the elliptic operator $\mathcal{L}f$ by
\begin{equation}\label{L}
\LL f = \Delta f + \ve\lan X,\grad f\ran.
\end{equation}
The next result is a direct consequence of Proposition \ref{L-prop-0}, and gives us the main equations to prove our results.

\begin{corollary}\label{L-prop}
Let $M^{n+k}$ be a $(n+k)$-dimensional Riemannian manifold which admits a conformal vector field $X$ with conformal factor $\vp.$ Let $\Sigma^n$ be a submanifold of $M^{n+k}$ such that the mean curvature vector ${\bf H}$ of $\Sigma^n$ satisfies ${\bf H}=\ve X^\perp$ for some $\ve\in\real,$ and $\{\eta,\eta_2,\ldots,\eta_k\}$ be an orthonormal frame of the normal bundle $(T\Sigma^n)^\perp\subset TM^{n+k},$ where, for $k=1,$ $\eta=N,$ the globally defined unitary normal vector field, and for $k\geq2,$ we assume that $X^\perp\neq 0$ and take $\eta=X^\perp/\|X^\perp\|.$ If $f=\lan X,\eta\ran,$ then 
\begin{equation}\label{L-eq0}
\LL f + (\|A\|^2+\ve\vp)f = -\sum_{\beta=2}^k s_{1\beta}(A_\beta X^\top) + \sum_{i=1}^n \lan\overline{R}(e_i,X^\top)e_i,\eta\ran.
\end{equation}
Here, $A$ and $A_\beta$ are the shape operators relative to the normals $\eta$ and $\eta_\beta,$ $\beta\in\{2,\ldots,k\},$ respectively, $s_{1\beta}(U)=\lan\n^\perp_U \eta,\eta_\beta\ran,$ $\overline{R}$ is the curvature tensor of $M^{n+k},$ and $\{e_1,\ldots,e_n\}$ is an orthonormal frame of $T\Sigma^n.$ Moreover, if $f\neq0,$ then
\begin{equation}\label{L-eq}
\LL f + (\|A\|^2+\ve\vp)f = \frac{1}{f}\sum_{\beta=2}^k\|A_\beta X^\top\|^2 + \sum_{i=1}^n \lan\overline{R}(e_i,X^\top)e_i,\eta\ran.
\end{equation}
In particular, if $M^{n+k}$ has constant sectional curvature and $\n^\perp\eta=0,$ or the immersion has codimension one and $M^{n+1}$ is Einstein, then
\begin{equation}\label{L-eq2}
\LL f + (\|A\|^2+\ve\vp)f =0.
\end{equation}
\end{corollary}
\begin{proof}
By using (\ref{trace1}) and (\ref{L}) in (\ref{lap-f0}), p.\pageref{lap-f0}, we obtain (\ref{L-eq0}). Equation (\ref{L-eq}) comes from replacing (\ref{x-norm-2}) in the first term of the right hand side of (\ref{L-eq0}). To prove (\ref{L-eq2}), notice that, if $M^{n+k}$ has constant sectional curvature $\kappa_0$, then
\[
\lan\overline{R}(e_i,X^\top)e_i,\eta\ran=\kappa_0(\lan e_i,e_i\ran\lan X^\top,\eta\ran - \lan X^\top, e_i\ran\lan \eta,e_i\ran)=0,
\]
since $\lan X^\top,\eta\ran=0=\lan \eta,e_i \ran.$ Moreover, if $\n^\perp\eta=0,$ then $s_{1\beta}\equiv 0$ for every $\beta\in\{2,\ldots,k\},$ i.e., $A_\beta X^\top=0.$ On the other hand, if $\ric_M = \lambda\lan\cdot,\cdot\ran,$ $\lambda\in\real,$ then
\[
\begin{aligned}
\sum_{i=1}^n \lan\overline{R}(e_i,X^\top)e_i,\eta\ran &= \ric_M(X^\top,\eta) - \lan\overline{R}(\eta,X^\top)\eta,\eta\ran\\
&= \ric_M(X^\top,\eta)=\lambda\lan X^\top,\eta\ran=0.
\end{aligned}
\]
\end{proof}
In order to prove our results, we will also need the classical Hopf maximum principle for elliptic operators:

\begin{lemma}[Hopf's maximum principle, see \cite{Serrin}]\label{hopf}
Let
\[
Lu = \sum_{i,j=1}^n a_{ij}(x)\frac{\partial^2 u}{\partial x_i \partial x_j} + \sum_{i=1}^n b_i(x)\frac{\partial u}{\partial x_i} + c(x)u
\]
be a strictly elliptic differential operator defined in a open set $\Omega\subset\real^n.$
\begin{itemize}
\item[(i)] If $c=0$, $Lu\geq 0$ (resp. $Lu\leq0$) and there exists $\max_\Omega u$ (resp. $\min_\Omega u$), then $u$ is constant.
\item[(ii)] If $c\leq 0,$ $Lu\geq 0$ (resp. $Lu\leq0$) and there exists $\max_\Omega u\geq 0$ (resp. $\min_\Omega u\leq 0$), then $u$ is constant.
\item[(iii)] Independently of the signal of $c,$ if $Lu\geq 0$ (resp. $Lu\leq0$) and $\max_\Omega u= 0$ (resp. $\min_\Omega u=0$), then $u$ is constant.
\end{itemize}
\end{lemma}
\section{Proof of the main theorems}

Now we are ready to proof our main theorems.

\begin{proof}[Proof of Theorem \ref{theo-1}]
In $\real^{n+1},$ the position vector is a conformal vector field with conformal factor $\vp=1.$ Since $\Sigma^n$ is a self-shrinker, we have
\begin{equation}\label{shrinker-0}
H=-\frac{1}{2}\lan X,N\ran=-\frac{1}{2}f,
\end{equation}
where $N$ is a unitary normal vector field. Since the codimension is one and $\Sigma^n$ is a self-shrinker, using Equation (\ref{L-eq2}) of Proposition \ref{L-prop} for $\ve=-1/2,$ we obtain
\begin{equation}\label{L-shrinker}
\LL f + \left(\|A\|^2-\frac{1}{2}\right)f=0.
\end{equation}
%
%
%
By using Newton's inequality 
\[
\frac{(\tr A)^2}{n}\leq \|A\|^2,
\]
Equation (\ref{shrinker-0}), and the hypothesis $\|A\|^2\leq 1/2$, we have

\begin{equation}\label{eqn-a}
f^2=4H^2=4(\tr A)^2 \leq 4n\|A\|^2\leq 2n.
\end{equation}
This implies
\[
-\sqrt{2n}\leq f\leq \sqrt{2n}
\]
and, thus, there exist $m=\inf_{\Sigma^n} f$ and $d=\sup_{\Sigma^n} f.$ If $m\leq 0,$ then
\[
\LL(f-m) + \left(\|A\|^2-\frac{1}{2}\right)(f-m)=-\left(\|A\|^2-\frac{1}{2}\right)m\leq 0.
\]
If $m=\min_{\Sigma^n} f,$ i.e., if $f$ reaches a minimum, then by the Hopf maximum principle (Lemma \ref{hopf}, item (ii)), applied to $f-m,$ we can conclude that $f$ is constant. On the other hand, if $m>0,$ then $d=\sup_{\Sigma^n} f>0.$ Thus if $f$ reaches (positive) a maximum, i.e., $d=\max_{\Sigma^n}f$ then, applying the Hopf maximum principle (Lemma \ref{hopf}, item (ii)), to equation (\ref{L-shrinker}), we conclude that $f$ is constant. 

On the other hand, Dajczer and Tojeiro, see \cite{DT}, Theorem 1, p.296, proved that the only hypersurfaces of $\real^{n+1}$ with constant support function $f$ are the cylinders, spheres and hyperplanes. The conclusion that $\Sigma^n=\s^{p}(\sqrt{2p})\times\real^{n-p},$ $0\leq p \leq n,$ comes from Equation (\ref{shrinker-0}). 

Thus we need to prove only that $f$ reaches a minimum. The proof that $f$ reaches a maximum is identical.

Since $W\neq\emptyset,$ there exists $p_0\in W,$ i.e., $p_0\not\in\bigcup_{p\in\Sigma^n}T_p\Sigma^n,$ which implies that $p-p_0\not\in T_p\Sigma^n$ for every $p\in\Sigma^n.$ Let $\{p_k\}$ be a sequence of points in $\Sigma^n$ such that $f(p_k)\to m$ when $k\to\infty.$ For each $p_k$ consider $q_k$ the projection of $p_0$ over $T_{p_k}\Sigma^n$ (see Figure \ref{Figura2}). 
\begin{figure}[ht]
	\centering
		\includegraphics[scale=0.7]{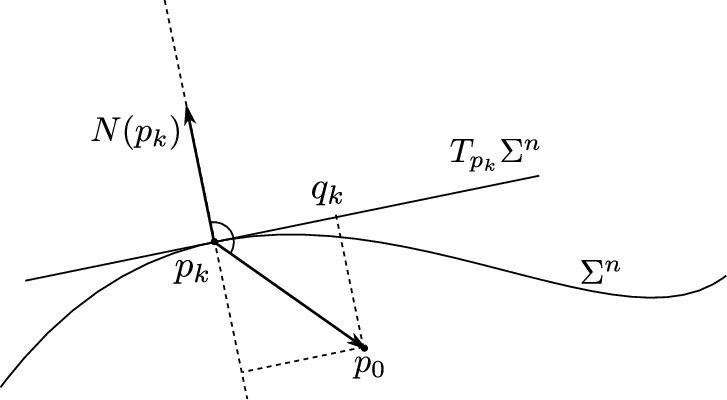}
		\caption{Projection of $p_0$ over $T_{p_k}\Sigma^n$}
\label{Figura2}
\end{figure}

Since
\[
\begin{aligned}
\dist(q_k,p_0)&=\|q_k-p_0\|=\| \proj_{N(p_k)}(p_k-p_0)\|=\lvert\lan p_k-p_0,N(p_k)\rvert\\
&= \lvert f(p_k) - \lan p_0,N(p_k)\ran\rvert\leq \lvert f(p_k)\rvert+\lvert p_0\rvert\\
&\leq \sqrt{2n} + \lvert p_0\rvert
\end{aligned}
\]
where $\proj_uv$ denotes the projection of the vector $v$ over the vector $u,$ we have that $\{q_k\}$ is a bounded sequence in $\bigcup_{p\in\Sigma^n}T_p\Sigma^n=\real^{n+1}-W.$ Moreover, since $W$ is open, we have that $\real^{n+1}-W$ is closed. Thus, passing to a subsequence if necessary, we can deduce that ${q_k}$ converges to a point $q_1\in \real^{n+1}-W.$ Let $p_1\in\Sigma^n$ such that $q_1\in T_{p_1}\Sigma^n.$ This implies
\[
f(p_1)=\lim_{k\to\infty} f(p_k)=m,
\]
i.e, $m$ is a minimum for $f.$
\end{proof}

\begin{proof}[Proof of Theorem \ref{theo-2-1}]

Since the codimension is one and $\Sigma^n$ is a self-shrinker, using Equation (\ref{L-eq2}) of Proposition \ref{L-prop} for $\ve=-1/2,$ we obtain
\[
\LL f = \left(\frac{1}{2}-\|A\|^2\right)f.
\]
Since $0\in W,$ we have that $0\not\in\bigcup_{p\in\Sigma^n}T_p\Sigma^n$ which implies that $p\not\in T_p\Sigma^n$ (seen as a vector centered at the origin) for every $p\in\Sigma^n.$ This implies that $f(p)=\lan p,N\ran\neq 0$ for every $p\in\Sigma^n,$ where $N$ denotes the unit normal vector field of $\Sigma^n$ in $\real^{n+1}.$ We assume, without loss of generality, that $f<0$ everywhere in $\Sigma^n.$ Thus, if $\|A\|^2\geq 1/2,$ then $\LL f\geq 0.$  Since $f<0,$ there exists $d=\sup_{\Sigma^n} f.$ Thus, if $f$ reaches a maximum, i.e., $d=\max_{\Sigma^n}f,$ then by using the Hopf maximum principle (Lemma \ref{hopf}, item (i)), we conclude that $f$ is constant. Therefore, we need to prove only that $f$ reaches a maximum.

Let $\{p_k\}$ be a sequence of points in $\Sigma^n$ such that $f(p_k)\to d$ when $k\to\infty.$ For each $p_k$ consider $q_k$ the projection of $p_0$ over $T_{p_k}\Sigma^n.$ Since
\[
\dist(q_k,0)=\lvert q_k\rvert=\lvert \proj_{N(p_k)} (p_k)\rvert=\lvert\lan p_k,N(p_k)\ran\rvert=-f(p_k)
\]
and $f(p_k)$ is a bounded sequence (since it converges), we have that $\{q_k\}$ is a bounded sequence in $\bigcup_{p\in\Sigma^n}T_p\Sigma^n=\real^{n+1}-W.$ Moreover, since $W$ is open, we have that $\real^{n+1}-W$ is closed. Thus, passing to a subsequence if necessary, we can deduce that ${q_k}$ converges to a point $q_1\in \real^{n+1}-W.$ Let $p_1\in\Sigma^n$ such that $q_1\in T_{p_1}\Sigma^n.$ This implies
\[
f(p_1)=\lim_{k\to\infty} f(p_k)=d,
\]
i.e, $d$ is a maximum for $f.$ Thus, $f=d$ is constant, which implies 
\[
\left(\frac{1}{2}-\|A\|^2\right)d=0.
\]
If $d=0,$ then $H=\frac{1}{2}\lan X,N\ran=0,$ which gives that $\Sigma^n$ is a hyperplane, but it contradicts the assumption $\|A\|^2\geq 1/2.$ Thus, $d<0,$ $\|A\|^2=1/2,$ and $\Sigma^n=\s^{p}(\sqrt{2p})\times\real^{n-p},$ $1\leq p \leq n,$ as in the proof of in Theorem \ref{theo-1}.
\end{proof}

\begin{proof}[Proof of Theorem \ref{theo-2}]
Since $\n^\perp \eta=0,$ where $\eta={\bf H}/\|{\bf H}\|,$ then $s_{1\beta}\equiv0,$ which implies that $A_\beta X^\top=0$ for every $\beta=2,\ldots,k.$ Since $\tr A_\beta=0$ and the dimension is two, we have that $A_\beta=0.$

On the other hand, in $\real^{2+k}$ the position vector is a conformal vector with conformal factor $\vp=1.$ In this case, since the codimension is $k\geq2$ and $\Sigma^2$ is a self-shrinker, we have
\[
f=\lan X,\eta\ran = \|X^\perp\|=2\|{\bf H}\|>0.
\]
Thus, by the Proposition \ref{L-prop}, we have, for $c=-1/2,$
\begin{equation}
\LL f = \left(\frac{1}{2}-\|A\|^2\right)f.
\end{equation}
\begin{itemize}
\item[{\bf i)}] If $\|A\|^2\leq 1/2,$ then $\LL f\geq0.$ Since $f^2\leq 4$ (see estimate (\ref{eqn-a}) in the proof of Theorem \ref{theo-1}), there exists $d=\sup_{\Sigma^2} f.$ Thus, if $f$ reaches a maximum, i.e., $d=\max_{\Sigma^2}f,$ then by using the Hopf maximum principle (Lemma \ref{hopf}, item (i)), we conclude that $f$ is constant. 

\item[{\bf ii)}] If $\|A\|^2\geq 1/2$ then $\LL f\leq 0.$ Since $f>0,$ there exists $m=\inf_{\Sigma^2} f.$ Thus, if $f$ reaches a minimum, i.e., $m=\min_{\Sigma^2}f,$ then by using the Hopf maximum principle, (Lemma \ref{hopf}, item (i)), we conclude that $f$ is constant.
\end{itemize}
The proof that $f$ reaches a maximum or a minimum is identical to that presented in the proof of Theorem \ref{theo-1} and Theorem \ref{theo-2-1}.

Thus, in both cases, $f$ is constant, which implies that $\|A\|=0$ and $\Sigma$ is plane passing through the origin, or $\|A\|^2=1/2.$ Since the second fundamental $\alpha$ satisfies

\[
\begin{aligned}
\LL\|\alpha\|^2 &= 2\|\n \alpha\|^2 + \|\alpha\|^2- 2\sum_{\beta\neq\delta}\|[A_\beta,A_\delta]\|^2\\
&\quad - 2\sum_{\beta,\delta}\left(\sum_{i,j=1}^2\lan\alpha(e_i,e_j),\eta_\beta\ran\lan\alpha(e_i,e_j),\eta_\delta\ran\right)^2\\
\end{aligned}
\]
\[
\begin{aligned}
&= 2\|\n \alpha\|^2 + \|\alpha\|^2- 2\sum_{\beta\neq\delta}\|A_\beta\circ A_\delta - A_\delta\circ A_\beta\|^2\\
&\quad - 2\sum_{\beta,\delta}\left(\sum_{i,j=1}^2\lan A_\beta(e_i),e_j\ran\lan A_\delta(e_i),e_j\ran\right)^2\\
&= 2\|\n \alpha\|^2 + \|\alpha\|^2- 2\sum_{\beta\neq\delta}\|A_\beta\circ A_\delta - A_\delta\circ A_\beta\|^2\\
&\quad - 2\sum_{\beta,\delta}(\tr(A_\beta\circ A_\delta))^2\\
\end{aligned}
\]
(see \cite{DX}, p.5069, Eq. (2.5)) and $A_\beta=0,$ $\beta=2,\ldots,k,$ we have $\|\alpha\|=\|A\|$ and
\[
\LL\|A\|^2 = 2\|\n A\|^2 + \|A\|^2 - 2\|A\|^4.
\]
Thus $\|A\|^2=1/2$ implies that $\|\n A\|^2=0.$ Therefore $\Sigma^2$ is isoparametric and thus $\Sigma^2=\s^1(\sqrt{2})\times\real$ or $\Sigma^2=\s^2(2).$
\end{proof}

%
%
%

\begin{proof}[Proof of Theorem \ref{theo-4}]
If $k\geq2,$ $\Sigma^n$ is a self-expander such that $f=2\|{\bf H}\|>0,$ and ${\bf H}/\|{\bf H}\|$ is parallel, then, by Proposition \ref{L-prop},
\[
\LL f + \left(\|A\|^2+\frac{1}{2}\right)f=0.
\]
If $k=1$ and $0\in W$, then $f\neq0$ as in the proof of Theorem \ref{theo-2-1} and we can assume $f>0.$ Thus $\LL f\leq 0.$ Since $f$ is bounded below, there exists $m=\inf_{\Sigma^n} f.$ Since $W$ is open and $0\in W$, reasoning as in the proof of Theorem \ref{theo-2-1}, we can prove that $m$ is actually a minimum. Therefore, by the Hopf maximum principle, we can see that $f$ is constant, which implies
\[
\left(\|A\|^2+\frac{1}{2}\right)f=0,
\]
but it is impossible, since $f>0.$
\end{proof}


\begin{bibdiv}
\begin{biblist}
\bib{A-L}{article}{
   author={Abresch, U.},
   author={Langer, J.},
   title={The normalized curve shortening flow and homothetic solutions},
   journal={J. Differential Geom.},
   volume={23},
   date={1986},
   number={2},
   pages={175--196},
   issn={0022-040X},
   review={\MR{845704}},
}

\bib{AB}{article}{
   author={Alencar, Hil\'{a}rio},
   author={Batista, M\'{a}rcio},
   title={Hypersurfaces with null higher order mean curvature},
   journal={Bull. Braz. Math. Soc. (N.S.)},
   volume={41},
   date={2010},
   number={4},
   pages={481--493},
   issn={1678-7544},
   review={\MR{2737313}},
   doi={10.1007/s00574-010-0022-z},
}

\bib{A-F}{article}{
   author={Alencar, Hil\'{a}rio},
   author={Frensel, Katia},
   title={Hypersurfaces whose tangent geodesics omit a nonempty set},
   conference={
      title={Differential geometry},
   },
   book={
      series={Pitman Monogr. Surveys Pure Appl. Math.},
      volume={52},
      publisher={Longman Sci. Tech., Harlow},
   },
   date={1991},
   pages={1--13},
   review={\MR{1173029}},
}

\bib{AS}{article}{
   author={Arezzo, Claudio},
   author={Sun, Jun},
   title={Self-shrinkers for the mean curvature flow in arbitrary codimension},
   journal={Math. Z.},
   volume={274},
   date={2013},
   number={3-4},
   pages={993--1027},
   issn={0025-5874},
   review={\MR{3078255}},
   doi={10.1007/s00209-012-1104-y},
}

\bib{BMR}{article}{
   author={Bianchini, Bruno},
   author={Mari, Luciano},
   author={Rigoli, Marco},
   title={Spectral radius, index estimates for Schr\"{o}dinger operators and
   geometric applications},
   journal={J. Funct. Anal.},
   volume={256},
   date={2009},
   number={6},
   pages={1769--1820},
   issn={0022-1236},
   review={\MR{2498559}},
   doi={10.1016/j.jfa.2009.01.021},
}
\bib{Cao-Li}{article}{
   author={Cao, Huai-Dong},
   author={Li, Haizhong},
   title={A gap theorem for self-shrinkers of the mean curvature flow in
   arbitrary codimension},
   journal={Calc. Var. Partial Differential Equations},
   volume={46},
   date={2013},
   number={3-4},
   pages={879--889},
   issn={0944-2669},
   review={\MR{3018176}},
   doi={10.1007/s00526-012-0508-1},
}

\bib{C-P}{article}{
   author={Cheng, Qing-Ming},
   author={Peng, Yejuan},
   title={Complete self-shrinkers of the mean curvature flow},
   journal={Calc. Var. Partial Differential Equations},
   volume={52},
   date={2015},
   number={3-4},
   pages={497--506},
   issn={0944-2669},
   review={\MR{3311901}},
   doi={10.1007/s00526-014-0720-2},
}

\bib{Cheng-Wei}{article}{
   author={Cheng, Qing-Ming},
   author={Wei, Guoxin},
   title={A gap theorem of self-shrinkers},
   journal={Trans. Amer. Math. Soc.},
   volume={367},
   date={2015},
   number={7},
   pages={4895--4915},
   issn={0002-9947},
   review={\MR{3335404}},
   doi={10.1090/S0002-9947-2015-06161-3},
}

\bib{CZ}{article}{
   author={Cheng, Xu},
   author={Zhou, Detang},
   title={Volume estimate about shrinkers},
   journal={Proc. Amer. Math. Soc.},
   volume={141},
   date={2013},
   number={2},
   pages={687--696},
   issn={0002-9939},
   review={\MR{2996973}},
   doi={10.1090/S0002-9939-2012-11922-7},
}

\bib{CVZ}{article}{
   author={Cheng, Xu},
   author={Vieira, Matheus},
   author={Zhou, Detang},
   title={Volume Growth of Complete Submanifolds in Gradient Ricci Solitons with Bounded Weighted Mean Curvature},
   journal={Internat. Math. Res. Notices},
   date={2019},
   issn={1073-7928},
   doi = {10.1093/imrn/rnz355},
}


\bib{Colding-Minicozzi}{article}{
   author={Colding, Tobias H.},
   author={Minicozzi, William P., II},
   title={Generic mean curvature flow I: generic singularities},
   journal={Ann. of Math. (2)},
   volume={175},
   date={2012},
   number={2},
   pages={755--833},
   issn={0003-486X},
   review={\MR{2993752}},
   doi={10.4007/annals.2012.175.2.7},
}


\bib{DT}{article}{
   author={Dajczer, Marcos},
   author={Tojeiro, Ruy},
   title={Hypersurfaces with a constant support function in spaces of
   constant sectional curvature},
   journal={Arch. Math. (Basel)},
   volume={60},
   date={1993},
   number={3},
   pages={296--299},
   issn={0003-889X},
   review={\MR{1201645}},
   doi={10.1007/BF01198815},
}
		
\bib{DX}{article}{
   author={Ding, Qi},
   author={Xin, Y. L.},
   title={The rigidity theorems of self-shrinkers},
   journal={Trans. Amer. Math. Soc.},
   volume={366},
   date={2014},
   number={10},
   pages={5067--5085},
   issn={0002-9947},
   review={\MR{3240917}},
   doi={10.1090/S0002-9947-2014-05901-1},
}

\bib{DXY}{article}{
   author={Ding, Qi},
   author={Xin, Y. L.},
   author={Yang, Ling},
   title={The rigidity theorems of self shrinkers via Gauss maps},
   journal={Adv. Math.},
   volume={303},
   date={2016},
   pages={151--174},
   issn={0001-8708},
   review={\MR{3552523}},
   doi={10.1016/j.aim.2016.08.019},
}

\bib{D-K}{article}{
   author={Drugan, Gregory},
   author={Kleene, Stephen J.},
   title={Immersed self-shrinkers},
   journal={Trans. Amer. Math. Soc.},
   volume={369},
   date={2017},
   number={10},
   pages={7213--7250},
   issn={0002-9947},
   review={\MR{3683108}},
   doi={10.1090/tran/6907},
}


\bib{Hall}{article}{
   author={Halldorsson, Hoeskuldur P.},
   title={Self-similar solutions to the curve shortening flow},
   journal={Trans. Amer. Math. Soc.},
   volume={364},
   date={2012},
   number={10},
   pages={5285--5309},
   issn={0002-9947},
   review={\MR{2931330}},
   doi={10.1090/S0002-9947-2012-05632-7},
}


\bib{halpern}{article}{
   author={Halpern, Benjamin},
   title={On the immersion of an $n$-dimensional manifold in
   $n+1$-dimensional Euclidean space},
   journal={Proc. Amer. Math. Soc.},
   volume={30},
   date={1971},
   pages={181--184},
   issn={0002-9939},
   review={\MR{286116}},
   doi={10.2307/2038246},
}

\bib{H-K}{article}{
   author={Hasanis, Thomas},
   author={Koutroufiotis, Dimitri},
   title={A property of complete minimal surfaces},
   journal={Trans. Amer. Math. Soc.},
   volume={281},
   date={1984},
   number={2},
   pages={833--843},
   issn={0002-9947},
   review={\MR{722778}},
   doi={10.2307/2000089},
}


\bib{Huisken}{article}{
   author={Huisken, Gerhard},
   title={Asymptotic behavior for singularities of the mean curvature flow},
   journal={J. Differential Geom.},
   volume={31},
   date={1990},
   number={1},
   pages={285--299},
   issn={0022-040X},
   review={\MR{1030675}},
}

\bib{K-M}{article}{
   author={Kleene, Stephen},
   author={M\o ller, Niels Martin},
   title={Self-shrinkers with a rotational symmetry},
   journal={Trans. Amer. Math. Soc.},
   volume={366},
   date={2014},
   number={8},
   pages={3943--3963},
   issn={0002-9947},
   review={\MR{3206448}},
   doi={10.1090/S0002-9947-2014-05721-8},
}
\bib{Li-Wei}{article}{
   author={Li, Haizhong},
   author={Wei, Yong},
   title={Classification and rigidity of self-shrinkers in the mean
   curvature flow},
   journal={J. Math. Soc. Japan},
   volume={66},
   date={2014},
   number={3},
   pages={709--734},
   issn={0025-5645},
   review={\MR{3238314}},
   doi={10.2969/jmsj/06630709},
}
\bib{MM}{article}{
   author={Mendon\c{c}a, S\'{e}rgio},
   author={Mirandola, Heudson},
   title={Hypersurfaces whose tangent geodesics do not cover the ambient
   space},
   journal={Proc. Amer. Math. Soc.},
   volume={136},
   date={2008},
   number={3},
   pages={1065--1070},
   issn={0002-9939},
   review={\MR{2361882}},
   doi={10.1090/S0002-9939-07-09282-9},
}
\bib{PRS}{article}{
   author={Pigola, Stefano},
   author={Rigoli, Marco},
   author={Setti, Alberto G.},
   title={Some applications of integral formulas in Riemannian geometry and
   PDE's},
   journal={Milan J. Math.},
   volume={71},
   date={2003},
   pages={219--281},
   issn={1424-9286},
   review={\MR{2120922}},
   doi={10.1007/s00032-003-0021-2},
}

\bib{Serrin}{article}{
   author={Pucci, Patrizia},
   author={Serrin, James},
   title={The strong maximum principle revisited},
   journal={J. Differential Equations},
   volume={196},
   date={2004},
   number={1},
   pages={1--66},
   issn={0022-0396},
   review={\MR{2025185}},
   doi={10.1016/j.jde.2003.05.001},
}
\bib{R}{article}{
   author={Rimoldi, Michele},
   title={On a classification theorem for self-shrinkers},
   journal={Proc. Amer. Math. Soc.},
   volume={142},
   date={2014},
   number={10},
   pages={3605--3613},
   issn={0002-9939},
   review={\MR{3238436}},
   doi={10.1090/S0002-9939-2014-12074-0},
}
\bib{Smoczyk}{article}{
   author={Smoczyk, Knut},
   title={Self-shrinkers of the mean curvature flow in arbitrary
   codimension},
   journal={Int. Math. Res. Not.},
   date={2005},
   number={48},
   pages={2983--3004},
   issn={1073-7928},
   review={\MR{2189784}},
   doi={10.1155/IMRN.2005.2983},
}


\bib{X-X}{article}{
   author={Xu, Hongwei},
   author={Xu, Zhiyuan},
   title={On Chern's conjecture for minimal hypersurfaces and rigidity of
   self-shrinkers},
   journal={J. Funct. Anal.},
   volume={273},
   date={2017},
   number={11},
   pages={3406--3425},
   issn={0022-1236},
   review={\MR{3706607}},
   doi={10.1016/j.jfa.2017.08.020},
}



\end{biblist}
\end{bibdiv}
\end{document}